\documentclass[10pt]{article}
\font\smallit=cmti10
\font\smalltt=cmtt10

\usepackage{amssymb,latexsym,amsmath,epsfig,amsthm} 

\usepackage{amsfonts,amsmath,amsthm,amssymb,amscd}

\makeatletter

\renewcommand\section{\@startsection {section}{1}{\z@}
{-30pt \@plus -1ex \@minus -.2ex}
{2.3ex \@plus.2ex}
{\normalfont\normalsize\bfseries\boldmath}}

\renewcommand\subsection{\@startsection{subsection}{2}{\z@}
{-3.25ex\@plus -1ex \@minus -.2ex}
{1.5ex \@plus .2ex}
{\normalfont\normalsize\bfseries\boldmath}}

\renewcommand{\@seccntformat}[1]{\csname the#1\endcsname. }

\newcommand\blfootnote[1]{
\begingroup
\renewcommand\thefootnote{}\footnote{#1}
\addtocounter{footnote}{-1}
\endgroup
}

\makeatother

\newtheorem{theorem}{Theorem}

\theoremstyle{definition}

\newcommand{\Ker}{\operatorname{Ker}}

\begin{document}
\begin{figure}
\vspace*{-62pt}
\vspace{.75in}
\end{figure}
\vspace*{-45pt}
\leftline{\smalltt\#A1} \vskip -12.5pt
\centerline{\smalltt  INTEGERS 26 (2026)}
\vskip 3pt \hrule																													
\begin{center}
\uppercase{\bf \boldmath Notes on the multiplier systems of $\eta(\tau)$ and $\theta(\tau)$}\blfootnote{DOI: }
\vskip 20pt
{\bf Kazuhide Matsuda}\\
{\smallit Faculty of Fundamental Science, 
National Institute of Technology (KOSEN), Niihama College,
Japan. }\\ 
{\tt ka.matsuda@niihama-nct.ac.jp}\\ 
\end{center}
\vskip 20pt
\centerline{\smallit Received: , Revised: , Accepted: , Published: } 
\vskip 30pt

\centerline{\bf Abstract}
\noindent
The multiplier systems of $\eta^{2k}(\tau)$ and $\theta^{2k}(\tau)$ $(k\in\mathbb{Z})$ are characters. 
In this paper, we determine their kernels $\Ker \nu_{\eta^{2k}}$ and $\Ker \nu_{ \theta^{2k} } $. 

\pagestyle{myheadings}
\markright{\smalltt INTEGERS: 26 (2026)\hfill}
\thispagestyle{empty}
\baselineskip=12.875pt
\vskip 30pt


\section{Introduction}
\label{intro}
The {\it upper half-plane} $H$ is defined as 
$
H=
\{
\tau\in\mathbb{C} \, | \,  \Im \tau>0
\}. 
$
Moreover, we set $q=\exp(2\pi i  \tau)$ and define 
\begin{equation*}
\displaystyle
\eta(\tau)=q^{\frac{1}{24}} \prod_{n=1}^{\infty} (1-q^n) \,\,\mathrm{and} \,\,
\theta(\tau)=\sum_{n=-\infty}^{\infty}  e^{\pi i n^2 \tau}. 
\end{equation*}
We adopt the definitions and the notation of modular groups and modular forms provided by Knopp throughout \cite{Knopp}. 
First, we define 
\begin{equation*}
\Gamma(1)=SL(2,\mathbb{Z}) \,\,
\mathrm{and} \,\,
\Gamma_{\theta}=
\left\{
\begin{pmatrix}
a & b \\
c & d
\end{pmatrix}
\in
\Gamma(1) \,\Big| \,
a\equiv d  \bmod{2} \,\,
\mathrm{and} \,\,
b \equiv c  \bmod{2} 
\right\}.
\end{equation*}
Mumford \cite[pp. 35]{Mumford} defined the group $\Gamma_{1,2}.$ We note that $\Gamma_{\theta}=\Gamma_{1,2}.$ 
A {\it modular group} $\Gamma$ is a subgroup of finite index in $\Gamma(1).$ 
A {\it modular form} of weight $r$ is a function $F(\tau)$ 
defined and meromorphic in $H$ 
that satisfies 
\begin{equation*}
F(M\tau)
=
\nu(M)
(c \tau+b)^r F(\tau), \,\,\tau\in H,
\end{equation*}
for all 
$
M
=
\begin{pmatrix}
a & b \\
c & d
\end{pmatrix}
\in \Gamma.
$
Here $\nu(M)$ is a complex number independent of $\tau$ such that $\,| \,\nu(M)\,| \,=1$ for all $M\in\Gamma.$ 
The function $\nu$ is called a {\it multiplier system}. 
We note that $\nu$ is a character if the weight $r$ is a rational integer.  
\par
It is well known that 
$\eta(\tau)$ is a modular form of weight $\displaystyle 1/2$ on $\Gamma(1)$ and that
$\theta(\tau)$ is a modular form of weight $\displaystyle 1/2$ on $\Gamma_{\theta}.$
Rademacher \cite{Rademacher} derived the asymptotic formula of the partition number $p(n)$ 
by using the multiplier system $\nu_{\eta}$ of $\eta(\tau).$ 
Moreover, in \cite[pp. 154]{Rademacher}, 
he showed that studying the Dedekind sums leads to the following modular group defined by 
\begin{equation*}
\left\{
\begin{pmatrix}
a & b \\
c & d
\end{pmatrix}
\in 
\Gamma(1) \, | \, 
\begin{pmatrix}
a & b \\
c & d
\end{pmatrix}
\equiv 
\pm
\begin{pmatrix}
1 & 0 \\
0 & 1
\end{pmatrix}, 
\pm
\begin{pmatrix}
0 & -1 \\
1 & 0
\end{pmatrix}, 
\pm
\begin{pmatrix}
1 & 1 \\
1 & -1
\end{pmatrix}, 
\pm
\begin{pmatrix}
-1 & 1 \\
1 & 1
\end{pmatrix}
\bmod{3}
\right\}. 
\end{equation*}
With the multiplier system $\nu_{\theta}$ of $\theta(\tau)$, 
Grosswald \cite{Grosswald} and Knopp \cite{Knopp} obtained 
an asymptotic formula of the number-theoretic function $r_s(n),$ 
where for $s$ and $n\in\mathbb{N},$ $r_s(n)$ is defined as 
\begin{equation*}
r_s(n)=
\sharp
\left\{
(x_1,x_2,\ldots,x_s)\in\mathbb{Z}^s \,|\, 
x_1^2+x_2^2+\cdots+x_s^2=n
\right\}. 
\end{equation*}
\par
By $\eta(\tau),$ 
Newman \cite{Newman} 
constructed modular functions with respect to $\Gamma_0(N)$. 
From modular functions, 
Ogg \cite{Ogg1,Ogg2, Ogg3} produced torsion points of finite order  on Jacobians of the modular curves $X_0(N).$ 
Ogg's work is a crucial ingredient in Mazur's theorem \cite{Mazur} on rational torsion points of elliptic curves. 
\par
From the multiplier system $\nu_{\eta}$, Rademacher \cite[pp. 180-182]{Rademacher} and Matsuda \cite{Matsuda} derived 
the transformation formula of the theta functions. 
\par 
The aim of this paper is to obtain new modular groups by studying the multiplier systems of $\eta(\tau)$ and $\theta(\tau)$, 
which leads to the discovery of nontrivial subgroups of the group $SL(2,\mathbb{Z}/N\mathbb{Z}) \, \,(N=2,3,4,6,12)$. 
For this purpose, we determine the kernels $\Ker \nu$ for the modular forms of integral weight $\eta^{2k}(\tau)$ and $\theta^{2k}(\tau)$ with $k\in \mathbb{Z}.$ 
We respectively consider $\eta(\tau)$ and $\theta(\tau)$ in sections \ref{sec:eta} and \ref{sec:theta}.

\section{The Multiplier System of $\eta(\tau)$}
\label{sec:eta}

Knopp \cite[pp.51]{Knopp} proved that for each 
$
M
=
\begin{pmatrix}
a & b \\
c & d
\end{pmatrix}
\in
\Gamma(1),
$
\begin{equation*}
\nu_{\eta}(M)=
\begin{cases}
\displaystyle
\left(
\frac{d}{c}
\right)^{*}   
\exp
\left\{
\frac{\pi i}{12}
[
(a+d)c-bd(c^2-1)-3c
]
\right\} 
&
\text{if $c$ is odd,}
\vspace{2mm}  \\
\displaystyle
\left(
\frac{c}{d}
\right)_{*}
\exp
\left\{
\frac{\pi i}{12}
[
(a+d)c-bd(c^2-1)+3d-3-3cd
]
\right\} 
&
\text{if $c$ is even,}
\end{cases}
\end{equation*}
where 
$
\displaystyle
\left(
\frac{d}{c}
\right)^{*}=\pm 1 
$ 
and 
$
\displaystyle
\left(
\frac{c}{d}
\right)_{*}=\pm 1.
$ 
Therefore, it follows that for each $k\in \mathbb{Z},$ 
\begin{equation*}
\nu_{\eta^{2k}  }(M)=
\begin{cases}
\displaystyle
\exp
\left\{
\frac{k\pi i}{6}
[
(a+d)c-bd(c^2-1)-3c
]
\right\} 
&
\text{if $c$ is odd,}
\vspace{2mm}  \\
\displaystyle
\exp
\left\{
\frac{k \pi i}{6}
[
(a+d)c-bd(c^2-1)+3d-3-3cd
]
\right\} 
&
\text{if $c$ is even.}
\end{cases}
\end{equation*}
For each 
$
M
=
\begin{pmatrix}
a & b \\
c & d
\end{pmatrix}
\in
\Gamma(1),
$
we define 
\begin{equation*}
f(M)=
\begin{cases}
\displaystyle
(a+d)c-bd(c^2-1)-3c 
&
\text{if $c$ is odd,}
\vspace{2mm}  \\
\displaystyle
(a+d)c-bd(c^2-1)+3d-3-3cd 
&
\text{if $c$ is even.}
\end{cases}
\end{equation*}
Below, we consider several specific cases within this system.

\subsection{The Case Where $k\equiv 0 \bmod{12}$}

From the formula of $\nu_{\eta^{2k}}(M),$ 
we can obtain the following theorem. 
\begin{theorem}
\label{thmk=0}
{\it
Suppose that 
$k\equiv 0 \bmod{12}$. 
Then, it follows that 
\begin{equation*}
\Ker \nu_{\eta^{2k}}
=\Gamma(1). 
\end{equation*}
}
\end{theorem}

\subsection{The Case Where $k\equiv 6 \bmod{12}$}

\begin{theorem}
\label{thmk=6}
{\it
Suppose that 
$k\equiv 6 \bmod{12}$. 
Then, it follows that 
\begin{equation*}
\Gamma_1=:
\Ker \nu_{\eta^{2k}}
=
\left\{
\begin{pmatrix}
a & b \\
c & d
\end{pmatrix}
\in
\Gamma(1) \,\Big| \, 
\begin{pmatrix}
a & b \\
c & d
\end{pmatrix}
\equiv
\begin{pmatrix}
1 & 0 \\
0 & 1
\end{pmatrix}, 
\begin{pmatrix}
0 & 1 \\
1 & 1
\end{pmatrix},
\begin{pmatrix}
1 & 1 \\
1 & 0
\end{pmatrix}
\bmod{2}
\right\}.
\end{equation*}
Moreover, 
as a coset decomposition of $\Gamma(1)$ modulo $\Ker \nu_{\eta^{2k}}$,  
we may select $S^{n}$ $(n=0,1),$ where 
$
S=
\begin{pmatrix}
1 & 1 \\
0 & 1
\end{pmatrix}. 
$
}
\end{theorem}

\begin{proof}[Proof of Theorem \ref{thmk=6}]
For each 
$
M
=
\begin{pmatrix}
a & b \\
c & d
\end{pmatrix}
\in
\Gamma(1), 
$ 
we have 
\begin{equation*}
\nu_{\eta^{2k}  }(M)=
\exp
\{\pi i f(M) \}. 
\end{equation*}
The group $SL(2,\mathbb{Z} / 2\mathbb{Z})$ consists of 
\begin{equation*}
\begin{pmatrix}
\bar{1} & \bar{0} \\
\bar{0} & \bar{1}
\end{pmatrix}, 
\begin{pmatrix}
\bar{1} & \bar{1} \\
\bar{0} & \bar{1}
\end{pmatrix}, 
\begin{pmatrix}
\bar{0} & \bar{1} \\
\bar{1} & \bar{0}
\end{pmatrix}, 
\begin{pmatrix}
\bar{0} & \bar{1} \\
\bar{1} & \bar{1}
\end{pmatrix}, 
\begin{pmatrix}
\bar{1} & \bar{1} \\
\bar{1} & \bar{0}
\end{pmatrix}, 
\begin{pmatrix}
\bar{1} & \bar{0} \\
\bar{1} & \bar{1}
\end{pmatrix}. 
\end{equation*}
The theorem can be established by substituting these into $f(M).$ 
The coset decomposition can be obtained by considering the values of $\nu_{\eta^{2k}  }(M).$
\end{proof}

\subsection{The Case Where $k\equiv \pm3 \bmod{12}$}

\begin{theorem}
\label{thmk=3}
{\it
Suppose that 
$k\equiv \pm3 \bmod{12}$. 
Then, it follows that 
\begin{align*}
\Gamma_2=:
\Ker \nu_{\eta^{2k}}
=
\Biggl\{
\begin{pmatrix}
a & b \\
c & d
\end{pmatrix}
\in
\Gamma(1) \,\Big| \, 
\begin{pmatrix}
a & b \\
c & d
\end{pmatrix}
\equiv&
\begin{pmatrix}
1 & 0 \\
0 & 1
\end{pmatrix}, 
\begin{pmatrix}
1 & 2 \\
2 & 1
\end{pmatrix},
\begin{pmatrix}
-1 & 0 \\
2 & -1
\end{pmatrix}, 
\begin{pmatrix}
-1 & 2 \\
0 & -1
\end{pmatrix},  \\
&
\begin{pmatrix}
2 & 1 \\
1 & 1
\end{pmatrix}, 
\begin{pmatrix}
2 & -1 \\
-1 & 1
\end{pmatrix}, 
\begin{pmatrix}
0 & 1 \\
-1 & -1
\end{pmatrix}, 
\begin{pmatrix}
0 & -1 \\
1 & -1
\end{pmatrix}, \\
&
\begin{pmatrix}
1 & 1 \\
1 & 2
\end{pmatrix}, 
\begin{pmatrix}
1 & -1 \\
-1 & 2
\end{pmatrix}, 
\begin{pmatrix}
-1 & 1 \\
-1 & 0
\end{pmatrix}, 
\begin{pmatrix}
-1 & -1 \\
1 & 0
\end{pmatrix}
\bmod{4}
\Biggr\}.
\end{align*}
Moreover, 
as a coset decomposition of $\Gamma(1)$ modulo $\Ker \nu_{\eta^{2k}}$,  
we may select 
\begin{equation*}
S^{n}  \,\, 
(0\leq n \leq3),  \,\,
S=
\begin{pmatrix}
1 & 1 \\
0 & 1
\end{pmatrix}. 
\end{equation*}
}
\end{theorem}

\begin{proof}[Proof of Theorem \ref{thmk=3}]
We first note that 
the coset decomposition can then be obtained by considering the values of $\nu_{\eta^{2k}  }(M).$
\par
For each 
$
M
=
\begin{pmatrix}
a & b \\
c & d
\end{pmatrix}
\in
\Gamma(1), 
$ 
we have 
\begin{equation*}
\nu_{\eta^{2k}  }(M)=
\exp
\left\{
\pm
\frac{
\pi i
}
{2} 
f(M) 
\right\},
\end{equation*}
which implies that 
\begin{equation*}
M
\in
\Ker \nu_{\eta^{2k}}
\Longleftrightarrow 
f(M)\equiv 0 \bmod{4}. 
\end{equation*}
If $f(M)\equiv 0\bmod{2},$  
by Theorem \ref{thmk=6}, 
we have 
\begin{equation*}
M
\equiv
\begin{pmatrix}
1 & 0 \\
0 & 1
\end{pmatrix}, 
\begin{pmatrix}
0 & 1 \\
1 & 1
\end{pmatrix},
\begin{pmatrix}
1 & 1 \\
1 & 0
\end{pmatrix}
\bmod{2}.
\end{equation*}
{\tt Case 1.} Suppose that 
$
M
\equiv \begin{pmatrix}
1 & 0 \\
0 & 1
\end{pmatrix} 
\bmod{2}.$
Since $ad-bc=1,$ we have 
\begin{equation*}
(a,d)\equiv (1,1), (-1,-1) \bmod{4},
\end{equation*}
which implies that 
\begin{equation*}
f(M)\equiv 0 \bmod{4} 
\Longleftrightarrow
bd+3d-3-3cd \equiv 0 \bmod{4}. 
\end{equation*}
\par
When $(a,d)\equiv (1,1)\bmod{4},$ it follows that 
$b+c\equiv 0 \bmod{4},$ which yields 
\begin{equation*}
(b,c)\equiv(0,0), (2,2) \bmod{4}.
\end{equation*}
\par
When $(a,d)\equiv (-1,-1)\bmod{4},$ it follows that 
$b+c\equiv 2 \bmod{4},$ which yields 
\begin{equation*}
(b,c)\equiv(0,2), (2,0) \bmod{4}.
\end{equation*}
{\tt Case 2.} Suppose that 
$
M
\equiv \begin{pmatrix}
0 & 1 \\
1 & 1
\end{pmatrix} 
\bmod{2},$ 
which implies that 
\begin{align*}
f(M)\equiv 0 \bmod{4} 
\Longleftrightarrow&
a+d\equiv 3 \bmod{4}    \\
\Longleftrightarrow& (a,d)\equiv (0,-1), (2,1) \bmod{4}.
\end{align*}
\par
Assume that $(a,d)\equiv (0,-1) \bmod{4}.$ 
Since $ad-bc=1,$ it follows that 
\begin{equation*}
(b,c)\equiv 
(1,-1), (-1,1) \bmod{4}.
\end{equation*}
\par
Assume that $(a,d)\equiv (2,1) \bmod{4}.$ 
Since $ad-bc=1,$ it follows that 
\begin{equation*}
(b,c)\equiv 
(1,1), (-1,-1) \bmod{4}.
\end{equation*}
{\tt Case 3.} Suppose that 
$
M
\equiv \begin{pmatrix}
1 & 1 \\
1 & 0
\end{pmatrix} 
\bmod{2},$ 
which implies that 
\begin{align*}
f(M)\equiv 0 \bmod{4} 
\Longleftrightarrow&
a+d\equiv 3 \bmod{4}    \\
\Longleftrightarrow& (a,d)\equiv (1,2), (-1,0) \bmod{4}.
\end{align*}
\par
Assume that $(a,d)\equiv (1,2) \bmod{4}.$ 
Since $ad-bc=1,$ it follows that 
\begin{equation*}
(b,c)\equiv 
(1,1), (-1,-1) \bmod{4}.
\end{equation*}
\par
Assume that $(a,d)\equiv (-1,0) \bmod{4}.$ 
Since $ad-bc=1,$ it follows that 
\begin{equation*}
(b,c)\equiv 
(1,-1), (-1,1) \bmod{4}.
\end{equation*}
\end{proof}

\subsection{The Case Where $k\equiv \pm 4 \bmod{12}$}

\begin{theorem}
\label{thmk=4}
{\it
Suppose that 
$k\equiv \pm4 \bmod{12}$. 
Then, it follows that 
\begin{align*}
\Gamma_3=:
\Ker \nu_{\eta^{2k}}
=
\Biggl\{
\begin{pmatrix}
a & b \\
c & d
\end{pmatrix}
\in
\Gamma(1) \,\Big| \, 
\begin{pmatrix}
a & b \\
c & d
\end{pmatrix}
\equiv&
\pm
\begin{pmatrix}
1 & 0 \\
0 & 1
\end{pmatrix}, 
\pm
\begin{pmatrix}
0 & -1 \\
1 & 0
\end{pmatrix},
\pm
\begin{pmatrix}
1 & 1 \\
1 & -1
\end{pmatrix}, 
\pm
\begin{pmatrix}
-1 & 1 \\
1 & 1
\end{pmatrix}
\bmod{3}
\Biggr\}.
\end{align*}
Moreover, 
as a coset decomposition of $\Gamma(1)$ modulo $\Ker \nu_{\eta^{2k}}$,  
we may select 
\begin{equation*}
S^{n}  \,\, 
(0\leq n \leq2),  \,\,
S=
\begin{pmatrix}
1 & 1 \\
0 & 1
\end{pmatrix}. 
\end{equation*}
}
\end{theorem}

\begin{proof}[Proof of Theorem \ref{thmk=4}]
For each 
$
M
=
\begin{pmatrix}
a & b \\
c & d
\end{pmatrix}
\in
\Gamma(1), 
$ 
we have 
\begin{equation*}
\nu_{\eta^{2k}  }(M)=
\exp
\left\{
\pm
\frac{
2 \pi i
}
{3} 
f(M) 
\right\},
\end{equation*}
which implies that 
\begin{equation*}
M
\in
\Ker \nu_{\eta^{2k}}
\Longleftrightarrow 
f(M)\equiv 0 \bmod{3}. 
\end{equation*}
\par
If $f(M)\equiv 0\bmod{3}$ and $c\equiv 0 \bmod{3},$   
it follows that $bd\equiv 0 \bmod{3}.$ 
Given that $ad-bc=1,$ we have $b\equiv 0 \bmod{3},$ 
which implies that  
\begin{equation*}
M
\equiv 
\pm 
\begin{pmatrix}
1 & 0 \\
0 & 1
\end{pmatrix} 
\bmod{3}.
\end{equation*}
\par
If $f(M)\equiv 0\bmod{3}$ and $c\not\equiv 0 \bmod{3},$ 
it follows that $a+d\equiv 0 \bmod{3},$ 
which implies that 
\begin{equation*}
(a,d)\equiv (0,0), (1,-1), (-1,1) \bmod{3}.
\end{equation*}
When $(a,d)\equiv(0,0) \bmod{3},$ 
it follows that $bc\equiv -1 \bmod{3},$ 
which implies that 
\begin{equation*}
M
\equiv 
\pm 
\begin{pmatrix}
0 & -1 \\
1 & 0
\end{pmatrix} 
\bmod{3}.
\end{equation*}
When $(a,d)\equiv(1,-1) \bmod{3},$ 
it follows that $bc\equiv 1 \bmod{3},$ 
which implies that 
\begin{equation*}
M
\equiv 
\begin{pmatrix}
1 & 1 \\
1 & -1
\end{pmatrix} , 
\begin{pmatrix}
1 & -1 \\
-1 & -1
\end{pmatrix} 
\bmod{3}.
\end{equation*}
When $(a,d)\equiv(-1,1) \bmod{3},$ 
it follows that $bc\equiv 1 \bmod{3},$ 
which implies that 
\begin{equation*}
M
\equiv 
\begin{pmatrix}
-1 & 1 \\
1 & 1
\end{pmatrix} , 
\begin{pmatrix}
-1 & -1 \\
-1 & 1
\end{pmatrix} 
\bmod{3}.
\end{equation*}
It may be readily demonstrated through direct calculation that 
the eight kinds of matrices satisfy $f(M)\equiv 0 \bmod{3}.$ 
The coset decomposition can then be obtained by considering the values of $\nu_{\eta^{2k}  }(M).$
\end{proof}

\subsection{The Case Where $k\equiv \pm 2 \bmod{12}$}

\begin{theorem}
\label{thmk=2}
{\it
Suppose that 
$k\equiv \pm2 \bmod{12}$. 
Then, it follows that 
\begin{align*}
\Gamma_4=:
\Ker \nu_{\eta^{2k}}
=
\Ker 
\nu_{ \eta^{12} }
\cap
\Ker \nu_{ \eta^8 }
=
\Gamma_1\cap\Gamma_3. 
\end{align*}
Moreover, 
as a coset decomposition of $\Gamma(1)$ modulo $\Ker \nu_{\eta^{2k}}$,  
we may select 
\begin{equation*}
S^{n}  \,\, 
(0\leq n \leq5),  \,\,
S=
\begin{pmatrix}
1 & 1 \\
0 & 1
\end{pmatrix}. 
\end{equation*}
}
\end{theorem}

\begin{proof}[Proof of Theorem \ref{thmk=2}]
For each 
$
M
=
\begin{pmatrix}
a & b \\
c & d
\end{pmatrix}
\in
\Gamma(1), 
$ 
we have 
\begin{equation*}
\nu_{\eta^{2k}  }(M)=
\exp
\left\{
\pm
\frac{
 \pi i
}
{3} 
f(M) 
\right\},
\end{equation*}
which implies that 
\begin{equation*}
M
\in
\Ker \nu_{\eta^{2k}}
\Longleftrightarrow 
f(M)\equiv 0 \bmod{6}
\Longleftrightarrow 
f(M)\equiv 0 \bmod{2} \,\,
\mathrm{and} \,\,
f(M)\equiv 0 \bmod{3}. 
\end{equation*}
The theorem follows from Theorems \ref{thmk=6} and \ref{thmk=4}. 
The coset decomposition can then be obtained by considering the values of $\nu_{\eta^{2k}  }(M).$
\end{proof}

\subsection*{Remark 1}
Theorem \ref{thmk=2} implies that the only parabolic point of $\Gamma_4$ is $\infty.$ 
Recall that 
any elliptic element of $\Gamma(1)$ is conjugate to 
\begin{equation*}
\pm
\begin{pmatrix}
0 & -1 \\
1 & 0
\end{pmatrix},
\pm
\begin{pmatrix}
0 & 1 \\
-1 & -1
\end{pmatrix},
\pm
\begin{pmatrix}
-1 & -1 \\
1 & 0
\end{pmatrix}.
\end{equation*}
Since the modular group $\Gamma_4$ is a normal subgroup of $\Gamma(1)$, 
it follows that 
$\Gamma_4$ has no elliptic elements. 
\par 
The smallest positive integer $\lambda$ such that $S^{\lambda}\in \Gamma_4$ is $\lambda=6.$ 
The multiplier system of $\eta^4(\tau)$ with respect to the modular group $\Gamma_4$ is trivial and 
\begin{equation*}
\eta^4(\tau)
=q^{\frac16}+\sum_{n=2}^{\infty} a_n q^{\frac{n}{6}}, \quad q=\exp(2\pi i \tau).
\end{equation*}
\par
Set the differential $\omega_{\eta^4}$ associated with $\eta^4(\tau).$ 
From Miyake \cite[pp.50-51]{Miyake}, 
the canonical divisor $(\omega_{\eta^4})$ is given by 
\begin{equation*}
(\omega_{\eta^4})=(1-1)P_{\infty}=0,
\end{equation*}
which implies that the modular curve $\Gamma_4\backslash H^{\ast}$ is an elliptic curve.

\subsection{The Case Where $k\equiv \pm 1, \pm 5 \bmod{12}$}

\begin{theorem}
\label{thmk=1and5}
{\it
Suppose that 
$k\equiv \pm1, \pm 5 \bmod{12}$. 
Then, it follows that 
\begin{align*}
\Gamma_5=:
\Ker \nu_{\eta^{2k}}
=
\Ker 
\nu_{ \eta^{6} }
\cap
\Ker \nu_{ \eta^8 }
=
\Gamma_2\cap\Gamma_3. 
\end{align*}
Moreover, 
as a coset decomposition of $\Gamma(1)$ modulo $\Ker \nu_{\eta^{2k}}$,  
we may select 
\begin{equation*}
S^{n}  \,\, 
(0\leq n \leq11),  \,\,
S=
\begin{pmatrix}
1 & 1 \\
0 & 1
\end{pmatrix}. 
\end{equation*}
}
\end{theorem}

\begin{proof}[Proof of Theorem \ref{thmk=1and5}]
For each 
$
M
=
\begin{pmatrix}
a & b \\
c & d
\end{pmatrix}
\in
\Gamma(1), 
$ 
we have 
\begin{equation*}
\nu_{\eta^{2k}  }(M)=
\exp
\left\{
\pm
\frac{
k
 \pi i
}
{6} 
f(M) 
\right\}, \,\,
(k,6)=1, 
\end{equation*}
which implies that 
\begin{equation*}
M
\in
\Ker \nu_{\eta^{2k}}
\Longleftrightarrow 
f(M)\equiv 0 \bmod{12}
\Longleftrightarrow 
f(M)\equiv 0 \bmod{4} \,\,
\mathrm{and} \,\,
f(M)\equiv 0 \bmod{3}. 
\end{equation*}
The theorem follows from Theorems \ref{thmk=3} and \ref{thmk=4}. 
The coset decomposition can then be obtained by considering the values of $\nu_{\eta^{2k}  }(M).$
\end{proof}

\subsection*{Remark 2}
Theorem \ref{thmk=1and5} implies that the only parabolic point of $\Gamma_5$ is $\infty.$ 
Since $\Gamma_5$ is a normal subgroup of $\Gamma(1),$ the modular group $\Gamma_5$ has no elliptic elements. 
Moreover, note that 
$
\begin{pmatrix}
-1 & 0 \\
0 & -1
\end{pmatrix}
\notin
\Gamma_5.
$
\par
The smallest positive integer $\lambda$ such that $S^{\lambda}\in \Gamma_5$ is $\lambda=12.$ 
Note that $\eta^2(\tau)$ and $\eta^4(\tau)$ are  modular forms of weight one and two with respect to $\Gamma_5$ 
and 
\begin{equation*}
\eta^2(\tau)=q^{\frac{1}{12}} 
\left(
1+\sum_{n=1}^{\infty} a_n q^{\frac{n}{12}}
\right) \,\, 
\mathrm{and} \,\,
\eta^4(\tau)=q^{\frac{2}{12}} 
\left(
1+\sum_{n=1}^{\infty} b_n q^{\frac{n}{12}}
\right), \,\,
q=\exp(2 \pi i \tau). 
\end{equation*} 
From Miyake \cite[pp. 52-53]{Miyake}, 
the canonical divisor $(\omega_{\eta^4})$ is given by 
\begin{equation*}
(\omega_{\eta^4})
=2 P_{\infty}, 
\end{equation*}
which implies that 
the genus of the modular curve $\Gamma_5 \backslash H^{\ast}$ is two.

\section{The Multiplier System of $\theta(\tau)$}
\label{sec:theta}

Knopp \cite[pp.51]{Knopp} proved that for each 
$
M
=
\begin{pmatrix}
a & b \\
c & d
\end{pmatrix}
\in
\Gamma_{\theta},
$
\begin{equation*}
\nu_{\theta}(M)=
\begin{cases}
\displaystyle
\left(
\frac{d}{c}
\right)^{*}   
e^{
-\pi i c/4
}
&
\text{if $b\equiv c\equiv 1,$ $a\equiv d\equiv 0 \,\, \bmod{2}$,}
\vspace{2mm}  \\
\displaystyle
\left(
\frac{c}{d}
\right)_{*}
e^{\pi i (d-1)/4} 
&
\text{if $a\equiv d\equiv 1,$ $b\equiv c\equiv 0 \,\, \bmod{2}$. }
\end{cases}
\end{equation*}
where 
$
\displaystyle
\left(
\frac{d}{c}
\right)^{*}=\pm 1 
$ 
and 
$
\displaystyle
\left(
\frac{c}{d}
\right)_{*}=\pm 1.
$ 
Therefore, it follows that for each $k\in \mathbb{Z},$ 
\begin{equation*}
\nu_{\theta^{2k}  }(M)=
\begin{cases}
\displaystyle
e^{
-k\pi i c/2
}
&
\text{if $b\equiv c\equiv 1,$ $a\equiv d\equiv 0 \,\, \bmod{2}$,}
\vspace{2mm}  \\
\displaystyle
e^{k \pi i (d-1)/2} 
&
\text{if $a\equiv d\equiv 1,$ $b\equiv c\equiv 0 \,\, \bmod{2}$. }
\end{cases}
\end{equation*}
For each 
$
M
=
\begin{pmatrix}
a & b \\
c & d
\end{pmatrix}
\in
\Gamma_{\theta},
$
we define 
\begin{equation*}
g(M)=
\begin{cases}
-c
&
\text{if $b\equiv c\equiv 1,$ $a\equiv d\equiv 0 \,\, \bmod{2}$,}
\vspace{2mm}  \\
d-1
&
\text{if $a\equiv d\equiv 1,$ $b\equiv c\equiv 0 \,\, \bmod{2}$. }
\end{cases}
\end{equation*}

Below, we consider several specific cases within this system.

\subsection{The Case Where $k\equiv 0 \bmod{4}$}

From the formula of $\nu_{\theta^{2k}}(M),$ 
we can obtain the following theorem. 
\begin{theorem}
\label{thm:theta-k=0mod4}
{\it
Suppose that 
$k\equiv 0 \bmod{4}$. 
Then, it follows that 
\begin{equation*}
\Ker \nu_{\theta^{2k}}
=\Gamma_{\theta}. 
\end{equation*}
}
\end{theorem}

\subsection{The Case Where $k\equiv 2 \bmod{4}$}

\begin{theorem}
\label{thm:theta-k=2mod4}
{\it
Suppose that 
$k\equiv 2 \bmod{4}$. 
Then, it follows that 
\begin{equation*}
\Ker \nu_{\theta^{2k}}
=
\left\{
\begin{pmatrix}
a & b \\
c & d
\end{pmatrix}
\in
\Gamma(1) \,\Big| \, 
\begin{pmatrix}
a & b \\
c & d
\end{pmatrix}
\equiv
\begin{pmatrix}
1 & 0 \\
0 & 1
\end{pmatrix}
\bmod{2}
\right\}.
\end{equation*}
Moreover, 
as a coset decomposition of $\Gamma_{\theta}$ modulo $\Ker \nu_{\theta^{2k}}$,  
we may select 
$
I   
$ 
and 
$
T, 
$ 
where 
$
I
=
\begin{pmatrix}
1 & 0 \\
0 & 1
\end{pmatrix}
$ 
and 
$
T
=
\begin{pmatrix}
0 & -1 \\
1 & 0
\end{pmatrix}.
$
}
\end{theorem}

\begin{proof}[Proof of Theorem \ref{thm:theta-k=2mod4}]
For each 
$
M
=
\begin{pmatrix}
a & b \\
c & d
\end{pmatrix}
\in
\Gamma_{\theta},
$ 
we have 
\begin{equation*}
\nu_{\theta^{2k}  }(M)=
\exp
\left\{
\pi i
g(M) 
\right\},
\end{equation*}
which implies that 
\begin{equation*}
M
\in
\Ker \nu_{\theta^{2k}}
\Longleftrightarrow 
g(M)\equiv 0 \bmod{2} 
\Longleftrightarrow 
d \equiv 1 \bmod{2}. 
\end{equation*}
The theorem follows from the definition of $\Gamma_{\theta}.$ 
The coset decomposition can then be obtained by considering the values of $\nu_{\theta^{2k}  }(M).$
\end{proof}

\subsection*{Remark 3}
The modular group $\Gamma(2)$ is a normal subgroup of $\Gamma(1)$ and has no elliptic elements.
As a coset decomposition of for $\Gamma(1)$ modulo $\Gamma(2),$ we may select 
\begin{equation*}
\begin{pmatrix}
1 & 0 \\
0 & 1
\end{pmatrix}, 
\begin{pmatrix}
1 & 1 \\
0 & 1
\end{pmatrix}, 
\begin{pmatrix}
0 & -1 \\
1 & 0
\end{pmatrix}, 
\begin{pmatrix}
0 & -1 \\
1 & 1
\end{pmatrix}, 
\begin{pmatrix}
-1& -1 \\
1 & 0
\end{pmatrix}, 
\begin{pmatrix}
1 & 0 \\
-1 & 1
\end{pmatrix},
\end{equation*}
which implies that the parabolic points of $\Gamma(2)$ are $\infty, 0,-1.$ 
From Knopp \cite[pp.48]{Knopp}, we recall that 
\begin{equation}
\label{eqn:theta-transformation}
\theta \left( -\frac{1}{\tau}  \right)
=
e^{-\pi i/4}\tau^{1/2}\theta(\tau), \,\,
\theta\left(-1-\frac{1}{\tau} \right)
=
\tau^{1/2} 
\sum_{n=0}^{\infty} a_n e^{2 \pi i[n+1/8]\tau}, \,\,a_0\neq 0, 
\end{equation}
which implies that 
\begin{equation*}
\theta^4 \left( -\frac1{\tau} \right)=-\tau^2 \theta^4(\tau), \,\,
\theta^4\left(-1-\frac{1}{\tau} \right)
=\tau^2
\sum_{n=0}^{\infty} 
b_n e^{2\pi i (n+1) \tau/2}, \,\,b_0\neq0.
\end{equation*}
\par
Set the differential $\omega_{\theta^4}$ associated with $\theta^4(\tau)$. 
The canonical divisor is given by 
\begin{equation*}
(\omega_{\theta^4})=
(0-1)P_{\infty}+(0-1)P_{0}+(1-1)P_{-1}
=-P_{\infty}-P_0.
\end{equation*}

\subsection{The Case Where $k\equiv \pm 1 \bmod{4}$}

\begin{theorem}
\label{thm:theta-k=pm1mod4}
{\it
Suppose that 
$k\equiv \pm 1 \bmod{4}$. 
Then, it follows that 
\begin{equation*}
\Ker \nu_{\theta^{2k}}
=
\left\{
\begin{pmatrix}
a & b \\
c & d
\end{pmatrix}
\in
\Gamma(1) \,\Big| \, 
a\equiv d\equiv 1 \bmod{4} \,\,
\mathrm{and} \,\,
b\equiv c\equiv 0 \bmod{2}
\right\}.
\end{equation*}
Moreover, 
as a coset decomposition of $\Gamma_{\theta}$ modulo $\Ker \nu_{\theta^{2k}}$,  
we may select 
$
\pm I, \,
\pm T
$ 
or 
$
T^n \,(0\leq n \leq 3) 
$
where 
$
I
=
\begin{pmatrix}
1 & 0 \\
0 & 1
\end{pmatrix}
$ 
and 
$
T
=
\begin{pmatrix}
0 & -1 \\
1 & 0
\end{pmatrix}.
$
}
\end{theorem}

\begin{proof}[Proof of Theorem \ref{thm:theta-k=pm1mod4}]
For each 
$
M
=
\begin{pmatrix}
a & b \\
c & d
\end{pmatrix}
\in
\Gamma_{\theta},
$ 
we have 
\begin{equation*}
\nu_{\theta^{2k}  }(M)=
\exp
\left\{
\pm 
\frac{
\pi i
}
{2
}
g(M) 
\right\},
\end{equation*}
which implies that 
\begin{equation*}
M
\in
\Ker \nu_{\theta^{2k}}
\Longleftrightarrow 
g(M)\equiv 0 \bmod{4} 
\Longleftrightarrow 
d \equiv 1 \bmod{4}. 
\end{equation*}
The theorem follows from the definition of $\Gamma_{\theta}.$ 
The coset decomposition can be then obtained by considering the values of $\nu_{\theta^{2k}  }(M).$
\end{proof}

\subsection*{Remark 4}
Set the modular group $\Gamma=\Ker \nu_{\theta^{2}  }.$
From Knopp \cite[pp.14]{Knopp}, recall that 
a set of coset representatives for $\Gamma(1)$ modulo $\Gamma_{\theta}$ is given by 
$I, S^{-1}, S^{-1}T.$ Therefore, it follows that a set of coset representatives for $\Gamma(1)$ modulo $\Gamma$ 
is given by 
\begin{equation*}
M=
\pm
\begin{pmatrix}
1& 0 \\
0 & 1
\end{pmatrix}, 
\pm
\begin{pmatrix}
1& -1 \\
0 & 1
\end{pmatrix}, 
\pm
\begin{pmatrix}
0 & -1 \\
1 & 0
\end{pmatrix}, 
\pm
\begin{pmatrix}
0& -1 \\
1 & -1
\end{pmatrix}, 
\pm
\begin{pmatrix}
-1& -1 \\
1 & 0
\end{pmatrix}, 
\pm
\begin{pmatrix}
-1& 0 \\
-1 & -1
\end{pmatrix}, 
\end{equation*}
which implies that the parabolic points of $\Gamma$ are $\infty, 0, \pm 1.$ 
\par
Recall that 
any elliptic element of $\Gamma(1)$ is conjugate to 
\begin{equation*}
\pm
\begin{pmatrix}
0 & -1 \\
1 & 0
\end{pmatrix},
\pm
\begin{pmatrix}
0 & 1 \\
-1 & -1
\end{pmatrix},
\pm
\begin{pmatrix}
-1 & -1 \\
1 & 0
\end{pmatrix}.
\end{equation*}
For each coset representative of $\Gamma(1)$ modulo $\Gamma$, 
we see that 
\begin{equation*}
\pm M
\begin{pmatrix}
0 & -1 \\
1 & 0
\end{pmatrix}
M^{-1} 
\notin \Gamma, \,\,
\pm M
\begin{pmatrix}
0 & 1 \\
-1 & -1
\end{pmatrix}
M^{-1} 
\notin \Gamma, \,\,
\pm M
\begin{pmatrix}
-1 & -1 \\
1 & 0
\end{pmatrix}
M^{-1} 
\notin \Gamma, \,\,
\end{equation*}
which implies that $\Gamma$ has no elliptic elements. 
\par
Equation (\ref{eqn:theta-transformation}) yields 
\begin{align*}
\theta\left(
\frac{-\tau}{-\tau-1}
\right)
=& \,
\theta
\left(
1-
\frac{1}{\tau+1}
\right)
=
(\tau+1)^{\frac12}
\sum_{n=0}^{\infty}
a_n 
e^{
2 \pi i 
[n+1/8]
(\tau+1)
}   \\
=& \, 
(\tau+1)^{\frac12}
\sum_{n=0}^{\infty}
a_n 
e^{  
\frac{\pi i}{4}
}
e^{
2 \pi i 
[n+1/8]
\tau
}.    
\end{align*}
\par
The smallest positive integer $\lambda$ such that $S^{\lambda}\in \Gamma$ is $\lambda=2.$ 
Note that $\theta^2(\tau)$ and $\theta^4(\tau)$ are  modular forms of weight one and two with respect to $\Gamma.$ 
\par
Set the differential $\omega_{\theta^4}$ associated with $\theta^4(\tau)$. 
Note that 
$
\begin{pmatrix}
-1 & 0 \\
0 & -1
\end{pmatrix}
\notin \Gamma. 
$
From Miyake \cite[pp. 52-53]{Miyake}, 
the canonical divisor is given by 
\begin{equation*}
(\omega_{\theta^4})=
P_1
+
P_{-1},
\end{equation*}
which implies that 
the genus of the modular curve $\Gamma \backslash H^{\ast}$ is two.

\subsection*{Acknowledgments}
This work was supported by JSPS KAKENHI Grant Number 22K03278. 
We thank the referee for recommending various improvements to the paper.




\end{document}